\documentclass[letter]{amsart}
\usepackage{amsmath, amsthm, amssymb}
\usepackage{hyperref}
\usepackage{cleveref}
\usepackage{enumitem}
\newtheorem{prop}{Proposition}

\newtheorem{theorem}[prop]{Theorem}
\newtheorem{cor}[prop]{Corollary}
\theoremstyle{definition}
\newtheorem{ex}{Example}
\newtheorem*{mydef}{Definition}

\newtheorem*{lemma*}{Lemma}

\crefname{prop}{proposition}{propositions}
\crefname{cor}{corollary}{corollaries}
\newcommand{\N}{\mathbb{N}}

\newcommand{\C}{\mathbb{C}}

\newcommand{\m}{\mathfrak{m}}
\newcommand{\n}{\mathfrak{n}}
\newcommand{\p}{\mathfrak{p}}
\newcommand{\q}{\mathfrak{q}}

\DeclareMathOperator{\Min}{Min}
\DeclareMathOperator{\Spec}{Spec}
\DeclareMathOperator{\mSpec}{mSpec}
\DeclareMathOperator{\Ht}{ht}
\DeclareMathOperator{\ara}{ara}

\DeclareMathOperator{\rad}{Rad}

\DeclareMathOperator{\Cl}{Cl}

\everymath{\displaystyle}
\setlist[enumerate,1]{label={\upshape(\alph*)}}

\begin{document}
\title{Infinite prime avoidance}
\author{Justin Chen}
\address{Department of Mathematics, University of California, Berkeley,
California, 94720 U.S.A}
\email{jchen@math.berkeley.edu}


\begin{abstract}
We investigate prime avoidance for an arbitrary set of prime ideals in a 
commutative ring. Various necessary and/or sufficient conditions for prime 
avoidance are given, which yield natural 
classes of infinite sets of primes that satisfy prime avoidance.
Examples and counterexamples are given 
throughout to illustrate the phenomena that can occur. 
As an application, we show how to use prime avoidance to construct 
counterexamples among rings essentially of finite type.
\end{abstract}

\maketitle

The classical prime avoidance lemma is one of the most
ubiquitous results in commutative algebra. Prime avoidance, along with finiteness of 
associated primes, is one of the basic building blocks of the theory of Noetherian rings. 
For example, the two results can be jointly used to choose generic nonzerodivisors 
(such as in the converse of Krull's Altitude Theorem), or to select a single annihilator 
for an ideal consisting of zerodivisors. 

As a fundamental technical result, the prime avoidance lemma has found various extensions
in the literature (cf. \cite{HH}, \cite{SV}). Moreover, special cases of infinite prime avoidance 
have in the past been used to great effect, perhaps most famously as a crucial step in 
Nagata's example of an infinite-dimensional Noetherian ring. This indicates the potential utility 
of understanding and applying infinite prime avoidance methodically. The goal of this note is to 
make initial steps in this direction. To this end, we first make a definition. For a commutative 
ring $R$ with $1 \ne 0$, $\Spec R$ denotes the set (for now) of prime ideals of $R$.

\begin{mydef}
Let $R$ be a ring, $\Lambda \subseteq \Spec R$. 
We say that $\Lambda$ \textit{satisfies prime avoidance} if 
$I \subseteq \bigcup_{\p \in \Lambda} \p \implies I \subseteq \p$ for some $\p \in \Lambda$,
for any $R$-ideal $I$.
\end{mydef}

Note that in the definition of prime avoidance, it is enough
to check the condition for prime ideals $I$, since ideals which are maximal with 
respect to being contained in $\bigcup_{\p \in \Lambda} \p$ (i.e. not meeting the
multiplicative set $R \setminus \bigcup_{\p \in \Lambda} \p$) exist by Zorn's lemma, 
and are prime. 

\begin{ex} \label{maxEx}
For any ring $R$, the set of maximal ideals $\mSpec R$ satisfies prime avoidance: if 
$I \subseteq \bigcup_{\m \in \mSpec R} \m$, then $I$ consists of nonunits, hence is 
contained in a maximal ideal. This example, though basic, is actually representative 
of all examples in some sense: cf. \Cref{mainThm}(3), (6).
\end{ex}

We now arrive at the classical prime avoidance lemma. 
For convenience we give a short direct proof (as opposed to one using induction):

\begin{lemma*}[Prime Avoidance] \label{palemma}
Let $R$ be a ring, $\Lambda \subseteq \Spec R$. If $\Lambda$ is finite, then 
$\Lambda$ satisfies prime avoidance.
\end{lemma*}

\begin{proof}
Write $\Lambda = \{ \p_1, \ldots, \p_n \}$. Suppose $I$ is an $R$-ideal such that 
$I \not \subseteq \p_i$ for any $i$, and choose $a_i \in I \setminus \p_i$. 
Removing redundant primes for the union, we may choose 
$b_i \in \p_i \setminus \bigcup_{j \neq i} \, \p_j$ for each $i$. Set 
$c_i := a_i \prod_{j \ne i} b_j$. Then $c_i \in \p_j$ iff $j \ne i$ by primeness 
of $\p_i$, so $c_1 + \ldots + c_n \in I \setminus \bigcup \p_i$.
\end{proof}

General though prime avoidance is, its single restriction is quite severe: the set 
$\Lambda$ must be finite! The proof above offers no recourse to relaxing this 
constraint. But it is not without good reason that this is the case, as 
prime avoidance may simply fail when $\Lambda$ is infinite,
even for sets of minimal primes:

\begin{ex} \label{nonNoethEx}
Let $k$ be a field, $R = k[x_0, x_1, \ldots]/(x_{2i}x_{2i+1} \mid i \ge 0)$. Then the
set of minimal primes $\Min(R)$ has cardinality $2^{\aleph_0}$: every minimal 
prime is of the form $(x_{a(i)} \mid i \ge 0)$ for a sequence $\{ a(i) \}_{i \ge 0}$ 
with $a(i) \in \{ 2i, 2i+1 \}$. Let $\p_{odd} := (x_1, x_3, \ldots)$
be the minimal prime of odd variables. 
Then $\p_{odd} \subseteq \bigcup_{\p \in \Min(R) \setminus \{\p_{odd}\}} \p$.

To see this, pick $f \in \p_{odd}$. Write $f$ as an $R$-linear combination of finitely
many generators of $\p_{odd}$, say $x_1, x_3, \ldots, x_{2j-1}$. Then e.g. $(x_1, 
x_3, \ldots, x_{2j-1}, x_{2j}, x_{2j+2}, \ldots)$
is a minimal prime of $R$ containing $f$ which is distinct from $\p_{odd}$.

By similar reasoning, every minimal prime of $R$ is contained in the union of the 
other minimal primes. We remark that in this ring, the set of all minimal primes 
does satisfy prime avoidance, but even this need
not hold in general: there exist reduced rings of dimension $> 0$ where every 
nonzerodivisor is a unit. 
\end{ex}

Even in much tamer rings, infinite prime avoidance need not hold. For instance,
Noetherian rings have only finitely many minimal primes, which prevents minimal 
primes from (mis)behaving as in \Cref{nonNoethEx}. However, in this setting the 
principal ideal theorem can sometimes force infinite prime avoidance to fail:

\begin{prop} \label{counterEx2}
Let $R$ be a Noetherian ring.

\begin{enumerate}
\item For any $\q \in \Spec R$, $\q \subseteq \bigcup_{\Ht \p \le 1} \p$ (so 
prime avoidance fails if $\Ht \q \ge 2$).

\item Suppose $R$ is also Jacobson. Then for any $\m \in \mSpec R$ with 
$\Ht \m \ge 2$, $\m \subseteq \bigcup_{\n \in \mSpec(R) \setminus \{\m\}} \n$.
\end{enumerate}
\end{prop}

\begin{proof}
(a): Pick $f \in \q$, and take $\p$ a minimal prime of $f$. Then $f \in \p$, and by Krull's 
Principal Ideal Theorem, $\Ht \p \le 1$.

(b): Pick $f \in \m$, and let $\p$ be a minimal prime of $f$ contained in $\m$ (i.e. the pullback
to $R$ of a minimal prime of $(R/(f))_\m$). Now $\Ht \p \le 1 \implies \p \ne \m \implies \p$ 
is not maximal; hence $\p$ is a (necessarily infinite) intersection of maximal ideals (as $R$ is
Jacobson). Thus there is a maximal ideal $\n \ne \m$ with $\p \subseteq \n$, so $f \in \n$.
\end{proof}

In spite of these examples, one can still ask for classes of infinite sets of 
primes which do satisfy prime avoidance. It turns out that this question does have 
some nice answers. Recall now that $\Spec R$ has the Zariski topology, 
with closed sets of the form $V(I) := \{ \p \mid I \subseteq \p \}$ for an $R$-ideal $I$, 
and a ring map $\varphi : R \to S$ induces a continuous map
$\varphi^* : \Spec S \to \Spec R$ via contraction. 

\begin{prop} \label{pullbackLemma}
Let $\varphi : R \to S$ be a ring map, which is either a surjection or a localization. If 
$\Lambda \subseteq \Spec S$ satisfies prime avoidance, then so does $\varphi^*(\Lambda)$.
\end{prop}

\begin{proof}
Let $\p \subseteq \bigcup_{\q \in \Lambda} \varphi^{-1}(\q)$. Then for all $x \in \p$, 
$x \in \varphi^{-1}(\q)$ for some $\q \in \Lambda$, i.e. $\varphi(x) \in \q$. Since $\varphi$ is 
either a surjection or a localization, any element of $\p S$ is of the form $s \cdot \varphi(x)$ 
for some $x \in \p$, $s \in S$, so this shows that $\p S \subseteq \bigcup_{\q \in \Lambda} \q$.
By prime avoidance of $\Lambda$, $\p S \subseteq \q$ for some $\q \in \Lambda$, hence
$\p \subseteq \varphi^{-1}(\p S) \subseteq \varphi^{-1}(\q)$. 
\end{proof}

We use \Cref{pullbackLemma} to give examples of infinite sets satisfying
prime avoidance. Hereafter when convenient, we view $\Spec(U^{-1}R)$ 
inside $\Spec R$ as $\{ \p \mid \p \cap U = \emptyset \}$.

\begin{cor} \label{locCor}
Let $R$ be a ring, $U \subseteq R$ a multiplicative set, and $I$ an $R$-ideal.
Then $V(I) \cap \Spec(U^{-1}R)$ satisfies prime avoidance.
\end{cor}

\begin{proof}
$\varphi : R \to U^{-1}(R/I)$ is a composite of localizations and 
surjections. Now apply \Cref{pullbackLemma} twice to
$V(I) \cap \Spec(U^{-1}R) = \varphi^*(\Spec(U^{-1}(R/I)))$.

Notice: this shows that both $V(I)$ and $\Spec(U^{-1}R)$ satisfy prime
avoidance (by taking $U = \{1\}$ and $I = 0$, respectively).
In addition, pulling back $\mSpec(U^{-1}(R/I))$ above gives that 
$V(I) \cap \mSpec(U^{-1}R) = \varphi^*(\mSpec(U^{-1}(R/I)))$ satisfies prime avoidance.
\end{proof}

\begin{ex} \label{epiCounterEx}
\Cref{pullbackLemma} may lead one to think that $\varphi^*(\Spec S)$ satisfies 
prime avoidance for any ring epimorphism $\varphi : R \to S$, but this is not true. Let 
$k = \overline{k}$ be a field, $\widetilde{R} = k[s,t,u]$, $S = k[x,y]$, and define 
$\widetilde{\varphi} : \widetilde{R} \to S$ by $s \mapsto x, t \mapsto xy, u \mapsto 
xy^2 - y$. Then $\widetilde{\varphi}$ induces $\varphi : R := \widetilde{R}/(su - t^2 + t) 
\to S$, which is a ring epimorphism. Since $R \cong k[x, xy, xy^2 - y] \subseteq S$, 
any nonunit in $R$ is also a nonunit in $S$. Thus if $\Lambda := \varphi^*(\Spec S)$, 
then $\m \subseteq \bigcup_{\p \in \Lambda} \p$ for any $\m \in \mSpec R$. However, 
$(s, t - 1, u)$ is a maximal ideal of $R$ that is not in $\Lambda$: if 
$s \in \varphi^{-1}(x - a, y - b)$, then $a = 0$, and then $\varphi(t) = xy 
= x(y - b) + bx \in (x, y - b) \implies t \in \varphi^{-1}(x, y - b)$.
\end{ex}

\begin{ex} \label{setOps}
It follows from \Cref{locCor} that basic Zariski-open sets (i.e. sets of 
the form $D(f) := (\Spec R) \setminus V(f)$ for some $f \in R$) satisfy prime avoidance. 
However, arbitrary Zariski-open sets need not: 
if $R = k[x,y]$ for $k$ a field, $\Lambda_1 := D(x)$, $\Lambda_2 := D(y)$, then $\Lambda_1 
\cup \Lambda_2 = (\Spec R) \setminus \{ (x, y) \}$ does not satisfy prime avoidance, by
\Cref{counterEx2}(b). This example also shows that the class of sets satisfying prime 
avoidance is neither closed under union nor taking complements in $\Spec R$.
\end{ex}

\begin{mydef}
Let $R$ be a ring. For $\Lambda \subseteq \Spec R$, define the following sets:

\begin{itemize}
\item $\Lambda_{max} := \{ \p \in \Lambda \mid \p \not \subseteq \q, 
\forall \q \in \Lambda \}$, the subset of maximal elements of $\Lambda$.
Notice: $\Lambda_{max}$ may be empty, even if $\Lambda$ is not!

\item $\Lambda_{cl} := \{ \q \in \Spec R \mid \exists \p \in \Lambda, 
\q \subseteq \p \}$, the downward-closure of $\Lambda$ in the poset 
$\Spec R$. Notice: $(\cdot)_{cl}$ is 
a closure operation (i.e. monotonic, increasing, and idempotent). Indeed,
$\Lambda_{cl} = \bigcup_{\p \in \Lambda} (\{\p\}_{cl}) = \bigcup_{\p \in \Lambda}
\Spec(R_\p)$.
\end{itemize}
\end{mydef}

These definitions allow for various characterizations of prime avoidance. 
For a ring map $\varphi : R \to S$, we say that $\varphi^*$ 
is surjective on closed points if $\mSpec R \subseteq \varphi^*(\Spec S)$ (or 
equivalently, $\mSpec R \subseteq \varphi^*(\mSpec S)$). In the following, keep in 
mind that although $W^{-1}I \subseteq W^{-1}J$ does not imply that $I \subseteq J$ 
in general, the implication does hold if $J$ is prime (and does not meet $W$).

\begin{theorem} \label{mainThm}
Let $R$ be a ring, $\Lambda \subseteq \Spec R$, 
$W := R \setminus \bigcup_{\p \in \Lambda} \p$. Then the following are equivalent:

\begin{enumerate}[label={\upshape(\arabic*)}]
\item $\Lambda$ satisfies prime avoidance

\item $\mSpec(W^{-1}R) \subseteq \Lambda_{max}$

\item $\mSpec(W^{-1}R) = \Lambda_{max}$

\item $\mSpec(W^{-1}R) \subseteq \Lambda_{cl}$

\item $\Spec(W^{-1}R) = \Lambda_{cl}$

\item There is a ring map $\varphi : R \to S$ such that 
\begin{enumerate}[leftmargin=*, align=left, label={\upshape(\roman*)}]
\item $\Lambda_{max} = \varphi^*(\mSpec S)$ (so $\exists$ induced map
$W^{-1}R \to S$), and
\item $\Spec S \to \Spec(W^{-1}R)$ is surjective on closed points
\end{enumerate}

\item $\Lambda_{cl}$ satisfies prime avoidance

\item $\Lambda_{max}$ satisfies prime avoidance and 
$\Lambda \subseteq (\Lambda_{max})_{cl}$.
\end{enumerate}
\end{theorem}

\begin{proof}
(1) $\iff$ (7): $\bigcup_{\p \in \Lambda} \p = \bigcup_{\p' \in \Lambda_{cl}} \p'$, and 
$I \subseteq \p$ for some $\p \in \Lambda$ iff 
$I \subseteq \p'$ for some $\p' \in \Lambda_{cl}$.

(1) $\implies$ (2): Take $\m \in \mSpec(W^{-1}R)$. Then $\m = W^{-1}\q$ where $\q \in \Spec R$
is maximal with respect to $\q \cap W = \emptyset$. By prime avoidance, $\q \subseteq \p$
for some $\p \in \Lambda$. But $\p \cap W = \emptyset$, so $\q = \p \in \Lambda_{max}$ by 
maximality of $\q$. 

(2) $\implies$ (3): Take $\p \in \Lambda_{max}$. Then $W^{-1}\p$ is a proper ideal in $W^{-1}R$,
so $W^{-1}\p \subseteq \m$ for some maximal ideal $\m \in \mSpec(W^{-1}R)$. By assumption,
$\m = W^{-1}\q$ for some $\q \in \Lambda_{max}$. Localizing further at $\q$ gives 
$\p R_\q \subseteq \q R_\q$ which implies $\p \subseteq \q$, so by maximality of $\p$ in 
$\Lambda$, $\p = \q$, hence $W^{-1}\p = \m \in \mSpec(W^{-1}R)$.

(3) $\implies$ (4): Clear.

(4) $\implies$ (5): Follows from $\Lambda \subseteq \Spec(W^{-1}R) = 
(\mSpec(W^{-1}R))_{cl} \subseteq \Lambda_{cl}$.

(5) $\implies$ (7): Follows from \Cref{locCor}.

(3) $\implies$ (6): Take $S := W^{-1}R$, $\varphi : R \to S$ the canonical map. Then (i) 
follows from (3), and (ii) is automatic.

(6) $\implies$ (2): Clear.

(3) + (5) $\implies$ (8): Clear.

(8) $\implies$ (1): $\Lambda_{max} \subseteq \Lambda \subseteq (\Lambda_{max})_{cl} 
\implies (\Lambda_{max})_{cl} = \Lambda_{cl}$. Now apply (7).  \qedhere
\end{proof}

\Cref{mainThm}(7) implies in particular that prime avoidance is determined by the 
downward-closed subsets of $\Spec R$, and for downward-closed sets, prime avoidance 
behaves well with intersections:

\begin{prop}
Let $R$ be a ring, $\{\Lambda_i\}$ a collection of downward-closed sets in $\Spec R$ 
satisfying prime avoidance. Then $\Lambda := \bigcap \Lambda_i$ is also 
downward-closed and satisfies prime avoidance.
\end{prop}

\begin{proof}
It is clear that $\Lambda$ is downward-closed. Let $\q \in \Spec R$, $\q \subseteq 
\bigcup_{\p \in \Lambda} \p \subseteq \bigcap_i \bigcup_{\p \in \Lambda_i} \p$. By 
prime avoidance of $\Lambda_i$, there exist $\p_i \in \Lambda_i$ such that $\q$ is 
contained in $\p_i$, for every $i$. Then $\q \in (\Lambda_i)_{cl} = \Lambda_i$ for 
every $i$, i.e. $\q \in \Lambda$.
\end{proof}

We can also give an analogue of \Cref{counterEx2}(b) in (co)dimension 1
(whose proof we postpone until after \Cref{allPAchar}):

\begin{prop} \label{DD}
Let $R$ be a Noetherian normal ring of dimension 1. 
\begin{enumerate}

\item For $\m \in \mSpec R$, $\mSpec(R) \setminus \{m\}$ satisfies prime avoidance
iff $[\m]$ is torsion in $\Cl R$ (the divisor class group of $R$).

\item Every $\Lambda \subseteq \Spec R$ satisfies prime avoidance iff $\Cl R$ is
a torsion group.

\end{enumerate}
\end{prop}

\Cref{DD}(b) naturally leads one to ask: what are the rings such that every set of primes 
satisfy prime avoidance? Such rings were introduced under the name of
compactly-packed (C.P.) rings in \cite{RV}, and have been fairly well-studied, e.g. in 
\cite{SW}, \cite{PS}. The condition which replaces torsion in the class group turns out 
to be that of arithmetic rank 1.
Recall that the arithmetic rank of an ideal $I$ is defined as
$\ara I := \inf\{n \mid \exists x_1, \ldots, x_n \in R, \sqrt{(x_1, \ldots, x_n)} = \sqrt{I} \}$. 

\begin{prop} \label{allPAchar}
Let $R$ be a ring. Then the following are equivalent:

\begin{enumerate}[label={\upshape(\arabic*)}]

\item For all $\Lambda \subseteq \Spec R$, $\Lambda$ satisfies prime avoidance

\item For all downward-closed $\Lambda \subseteq \Spec R$, $\Lambda$ satisfies 
prime avoidance

\item For all Zariski-open sets $U \subseteq \Spec R$, $U$ satisfies prime avoidance

\item For all $\q \in \Spec R$, $(\Spec R) \setminus V(\q)$ satisfies prime avoidance

\item For all $\q \in \Spec R$, $\ara \q \le 1$.
\end{enumerate}
\end{prop}

\begin{proof}
(1) $\implies$ (2) $\implies$ (3) $\implies$ (4): Clear.

(4) $\implies$ (5): Let $\q \in \Spec R$. Then $\q \not \subseteq \p$ for all $\p \not 
\in V(\q)$, so by prime avoidance $\q \not \subseteq 
\bigcup_{\p \not \in V(\q)} \p$. Thus there is $x \in \q \setminus
\bigcup_{\p \not \in V(\q)} \p$, and such an $x$ has $\q$ as its 
only minimal prime (if $x \in \p$ for some $\p \in \Spec R$, then $\p \in V(\q)$), i.e. 
$\sqrt{(x)} = \q$.

(5) $\implies$ (1): Let $\q \in \Spec R$, $\q \subseteq \bigcup_{\p \in \Lambda} \p$. 
By hypothesis $\q = \sqrt{(x)}$ for some $x \in R$. Then $x \in \p$ for some 
$\p \in \Lambda \implies \q = \sqrt{(x)} \subseteq \p$.
\end{proof}

\begin{proof}[Proof of \Cref{DD}]
If $R$ is Dedekind and $\m \in \mSpec R$, then $\ara \m = 1$ iff $[\m]$ is torsion in
$\Cl R$: by unique factorization of ideals, $\sqrt{(x)} = \m \iff (x) = \m^n$ for some
$n \in \N$. 
If now $R$ is any Noetherian normal ring of dimension 1, then $R$ is a finite product
of Dedekind domains and fields, so the above reasoning, along with $(1) \iff (4)$
in \Cref{allPAchar}, gives (a) and (b).
\end{proof}

It is shown in \cite{PS} that if $R$ is Noetherian, then (5) in \Cref{allPAchar} may be replaced
with ($5'$): For all $\m \in \mSpec R$, $\ara \m = 1$ (which implies $\dim R \le 1$, since
$\Ht I \le \ara I$ in a Noetherian ring). In other words, under these assumptions 
the minimal primes also have arithmetic rank 1. 
In general though, it is possible for a 
minimal prime to be contained in a union of height 1 primes not containing it:

\begin{ex} \label{nonDomainEx}
Let $k$ be a field, $R = k[x,y,z]/(xy,xz)$, and $\q := (y, z)$, the 
non-principal minimal prime of $R$. If $\Lambda =$ all height 1 primes not containing 
$\q$, then $\q \subseteq \bigcup_{\p \in \Lambda} \p$: to see this, take $0 \ne f \in \q$,
and let $f_1$ be an irreducible factor of $\overline{f}$ in $R/(x) \cong k[y,z]$. Then 
$(x, f_1)$ is a height 1 prime of $R$ containing $f$, but not $\q$. Together with the 
above reasoning, this shows that $\ara \q = 2$.
\end{ex}

There is another interesting characterization of the C.P. property for domains 
via overrings: a Dedekind domain $R$ is C.P. iff every overring
of $R$ (i.e. a ring $S$ with $R \subseteq S \subseteq \operatorname{Quot}(R)$) is 
a localization of $R$. Moreover, a Noetherian domain of dimension 1 is C.P. iff every 
sublocalization (i.e. an overring that is an intersection of localizations)
is a localization. See \cite{HR}, Corollaries 2.8 and 3.13 for more details.

\vspace{0.15cm}
It would also be remiss not to mention the geometric interpretation of prime 
avoidance, which is closer in spirit to the titular ``avoidance".
For an affine scheme $X = \Spec R$, a (prime) cycle in $X$ will mean an integral closed 
subscheme of $X$ (i.e. a subscheme of the form $V(\p)$ for some $\p \in \Spec R$). 
A set of cycles $\{ Z_i \}$ in $X$ satisfies prime avoidance iff for any cycle $Z$ not 
containing any $Z_i$, there is a hypersurface in $X$ containing $Z$ but not 
any $Z_i$. If the $Z_i$'s consist of closed points, then this may be restated as:
any cycle avoiding the $Z_i$ can be extended to a hypersurface avoiding the $Z_i$. 
One can use this to see that a set $\Lambda$ of closed points 
in $\mathbb{A}^2_k$ with $|\Lambda| < |k|$ satisfies prime avoidance:
if $p \not \in \Lambda$, then there are $\ge |k|$ lines through $p$, 
and $\le |\Lambda|$ of these can meet $\Lambda$. 
This includes e.g. any discrete ($=$ has no limit points)
set of points in $\mathbb{A}^2_{\mathbb{R}}$.


\vspace{0.2cm}
We conclude with some applications of the ideas of prime avoidance. In general, prime 
avoidance is a constraint on a set of primes which can be used to justify one's intuition 
about the set (one interpretation of \Cref{mainThm}(3) is that prime avoidance means 
there are no ``unexpected" closed points in the localization). In particular, prime avoidance
can be used to construct rings essentially of finite type satisfying given conditions.
Although the examples below are of independent interest,
we use prime avoidance to verify certain properties of each:

\begin{ex} \label{closedPtsEx}
We give an example of a reduced, connected
Noetherian affine scheme such that the closure of the closed points is a proper 
closed set of codimension 0. Algebraically, this is a Noetherian ring with no 
nilpotents or idempotents such that the Jacobson radical $\rad R$ is nonzero, but 
consists of zerodivisors. In other words, $\rad R$ lies strictly between 
the intersection and union of the minimal primes:
\[ \bigcap_{\p \in \Min R} \p \subsetneq 
\bigcap_{\m \in \mSpec R} \m \subsetneq
\bigcup_{\p \in \Min R} \p
\]
For the example: let $k = \overline{k}$ be a field, $T := k[x,y]/(xy)$, 
$\Lambda := V(x) \subseteq \Spec T$, $W := T \setminus 
\bigcup_{\p \in \Lambda} \p$, and $R := W^{-1}T$. 
By \Cref{locCor}, $\Lambda$ satisfies prime 
avoidance, so by \Cref{mainThm}, $\mSpec R = \{ W^{-1}(x, y - a) \mid a \in k \}$. 
Since $k$ is infinite, $\rad R = \bigcap_{a \in k} W^{-1}(x, y - a) = W^{-1}(x) \ne 0$,
and $x$ is a zerodivisor in $R$. 
\end{ex}







\begin{ex}
We give an example of a Jacobson ring $R$ with the property that 
every ring surjection $R \twoheadrightarrow S$ induces a surjection on units
$R^\times \twoheadrightarrow S^\times$. 
Let $T := \C[x]$, $\Lambda := \{ (x - n) \mid n \in \N\}$, 
$W := T \setminus \bigcup_{\p \in \Lambda} \p$, and $R := W^{-1}T$. Since 
$\Cl T = 0$, by \Cref{DD} every subset of $\Spec T$ satisfies prime avoidance, 
so by \Cref{mainThm} 
$\mSpec(R) = \{ W^{-1}(x - n) \mid n \in \N \}$,
hence $R$ is a 1-dimensional Jacobson PID.

If $\varphi : R \twoheadrightarrow S$ is surjective, set $I := \ker \varphi$. Then 
$I = (f)R = W^{-1}(f)$ for some $f \in T$. Suppose $\exists g \in T$ 
with $\frac{g}{1} \not \in R^\times$, but $\varphi \Big(\frac{g}{1} \Big) \in S^\times$. 
Since $S \cong R/I = W^{-1}T/W^{-1}(f) \cong T/(f)$, it suffices to 
show that $g + f_1 \in W$ for some $f_1 \in (f)$, i.e. $g + f_1$ has no
roots in $\N$. Since $f, g$ have no common roots, this is 
possible by taking $f_1 = cf^n$ where $c \in \C$, $n \in \N$ are 
such that $\deg f^n > \deg g$ and $|c| \gg 0$. 
\end{ex}

\end{document}